\documentclass[12pt]{elsarticle}

\usepackage{amssymb}
\usepackage{amsmath}
\usepackage{amsthm}
\usepackage{amscd}
\usepackage{enumerate}
\usepackage{hyperref}
\usepackage[usenames,dvipsnames]{color}
\usepackage{dsfont,txfonts}
\usepackage[margin=1in]{geometry}

\theoremstyle{plain}
\newtheorem{theorem}{Theorem}

\newtheorem{corollary}{Corollary}
\newtheorem{lemma}{Lemma}
\newtheorem{definition}{Definition}

\theoremstyle{definition}
\newtheorem{example}{Example}
\newtheorem{remark}{Remark}

\newcommand{\F}{\mathcal{F}}
\newcommand{\Poly}{\mathcal{P}}
\newcommand{\cgeom}{\mathcal{L}}
\DeclareMathOperator{\ExtHull}{ExtConv}
\DeclareMathOperator{\Hull}{Conv}
\DeclareMathOperator{\Pos}{Pos}
\DeclareMathOperator{\cdim}{cdim}

\begin{document}

\title{Embedding convex geometries and a bound on convex dimension\footnote{The authors would like to thank Kira Adaricheva and Paul Edelman for helpful comments and suggestions.}}

\begin{abstract}
	The notion of an abstract convex geometry, due to \citet*{EJ}, offers an abstraction of the standard notion of convexity in a linear space.  \citet*{KNO} introduce the notion of a generalized convex shelling into $\mathbb{R}^N$ and prove that a convex geometry may always be represented with such a shelling. We provide a new, shorter proof of their result using a representation theorem of \cite{EJ} and deduce a different upper bound on the dimension of the shelling. Furthermore, in the spirit of \citet*{C}, who shows that any 2-dimensional convex geometry may be embedded as circles in $\mathbb{R}^2$, we show that any convex geometry may be embedded as convex polygons in $\mathbb{R}^2$.
\end{abstract}

\author[mr]{Michael Richter\corref{cor}}
\address[mr]{Department of Economics, Royal Holloway, University of London, Egham, UK.}
\ead{michael.richter@rhul.ac.uk}
\author[lgr]{Luke G. Rogers}
\address[lgr]{Department of Mathematics, University of Connecticut, Storrs, CT, 06269-3009, USA.}
\ead{rogers@math.uconn.edu}
\cortext[cor]{Corresponding author}

\maketitle

\section{Introduction}

One of the clearest and earliest expositions of the theory of convex geometries is found in \cite{EJ}. There, one may find several equivalent definitions of what it means for a collection of sets to constitute an abstract convex geometry. One is given below in Definition \ref{defn:cgeom} stating that the collection of sets is i) closed under intersection, ii) includes both the empty set and the entire set, and iii) given any set in the collection, there is an element outside of it which may be added to that set and yields another member of the collection. Another equivalent formulation is given in terms of the anti-exchange property, which states that if one is given two elements and a set in the collection containing neither of those elements, then there is a larger set in the collection containing exactly one of those elements. That is, it is possible to include one of the external elements (perhaps along with additional other elements) without including the other. This is the same anti-exchange property that appears in antimatroid theory, and there is a clear equivalence between convex geometries and antimatroids. Additionally, \citet*{E80} shows that every convex geometry is a meet-distributive lattice and vice-versa. In fact, there are several other equivalent formulations (for surveys, see \citet*{S}, \citet*{AC}, \citet*{C14} and \citet*{AN}).  Finally, we note that convex geometries have proven useful in economics in the study of choice (\citet*{JD96,JD01,JD05}, \citet*{K}) and of abstract economic equilibrium (\citet*{RR}).

\medskip

In their expansive study of convex geometries, \citet*{AGT} provide a series of problems of which one, problem 3, inquires about subclasses of join-semidistributive lattices which are relatively embeddable into the lattice of convex sets of $\mathbb{R}^N$. Partial solutions were achieved by \citet*{A} and \citet*{WS} particularly for \emph{lower bounded} lattices. More recently, \cite{KNO} have provided a solution by showing that any convex geometry can be represented as a ``generalized convex shelling'', meaning that there is an embedding in $\mathbb{R}^N$ so that each set in the geometry is convex if and only if its embedding is convex with respect to a fixed external set of points in $\mathbb{R}^N$.  In this paper we provide an alternate proof of this representation theorem using a representation of \cite{EJ} which represents a convex geometry through a collection of orderings.  An important feature of this new proof is that it yields an upper bound on the dimension of the smallest Euclidean space into which a convex geometry may be embedded via a generalized convex shelling which may be different to the upper bound found by~\cite{KNO}. We then further use the representation theorem of \cite{EJ} to provide a new result, that any convex geometry may be embedded as convex polygons in $\mathbb{R}^2$.

\begin{definition}\label{defn:cgeom}
Let $E$ be a finite set containing $N$ points.  A convex geometry on $E$ is a collection $\cgeom$ of subsets (called convex sets) of $E$ with the following  properties.
\begin{enumerate}
\item $\emptyset$ and $E$ are in $\cgeom$
\item If $X,Y\in\cgeom$ then $X\cap Y\in\cgeom$
\item If $X\in\cgeom\setminus\{E\}$ then there is $e\in E\setminus X$ so $X\cup\{e\}\in\cgeom$\end{enumerate}
\end{definition}

Convex geometries $\cgeom_{1}$ on $E_{1}$ and $\cgeom_{2}$ on $E_{2}$ are isomorphic if there is a bijection $\psi :E_{1}\to E_{2}$ so $\psi(X)\in\cgeom_{2}$ if and only if $X\in\cgeom_{1}$.

\medskip

The following definition of a generalized convex shelling for a set $E\subset\mathbb{R}^{n}$ is due to \cite{KNO}.

\begin{definition}\label{defn:cvxshelling}
Using $\Hull(Q)$ to denote the convex hull, let $G,Q\subset\mathbb{R}^{n}$ be finite sets such that $G\cap\Hull(Q)=\emptyset$. The generalized convex shelling on $G$ with respect to $Q$ is
\begin{equation*}
	\cgeom = \bigl\{ X\subset G: \Hull(X\cup Q)\cap G=X\bigr\}.
	\end{equation*}
\end{definition}
It is easily verified that this defines a convex geometry.  In~\cite{KNO} the converse is proved.
\begin{theorem}[\protect{\cite{KNO} Theorem~2.8}]\label{thm:KNO}
Any convex geometry is isomorphic to a generalized convex shelling.
\end{theorem}

Another straightforward way to define a convex geometry on $E$ uses a collection of orderings, which we denote $\succcurlyeq_{i}$.  Throughout the paper, orderings are total and antisymmetric.

\begin{definition}\label{defn:cgeomfromordering}
We say that $\cgeom$ is generated by a family $\{\succcurlyeq_{i}\}_{i=1}^{M}$  of orderings on $E$ if
\begin{equation*}
	\cgeom = \{\emptyset\}\cup\{ X \subseteq E: \forall y \notin X, \ \exists i \text{ so that }\forall x \in X, x \succ_i y \}
	\end{equation*}
\end{definition}


Given any finite set $E$ of Euclidean space, the standard convexity notion on it can be  generated in the manner of Definition 3 by taking all strict orderings which are generated by linear directions. Then, given any convex set $X$ and any $y\in E\setminus X$ there is a Euclidean separating hyperplane, so by projecting onto the line orthogonal to this hyperplane one finds a linear order $\succ$ for which $x\in X$ implies $x\succ y$. This shows that all standard convex sets are convex in the geometry defined by these orderings. To see the converse, notice that in the generated convexity, each convex set is Euclidean convex because it is the intersection of the half-spaces defined by some collection of $\succ_i$ with $E$. Thus, the standard convexity can be represented by the family of orderings arising from linear directions and the following lemma shows that the above definition produces a convex geometry for any set and for any set of orderings. This was proven in \cite{EJ}, but for the reader's convenience we also demonstrate it here.

\begin{lemma}[\protect{\cite{EJ} Theorem 5.2}]\label{lem:LisomL1}
If $\cgeom$ is generated by a family of orderings $\{\succcurlyeq_{i}\}_{i=1}^{m}$, then it is a convex geometry.
\end{lemma}
\begin{proof}
The only non-trivial part of Definition~\ref{defn:cgeom} to verify is (3). Take $X\in\cgeom\setminus\{E\}$ and define a relation on $E \backslash X$ by  $a \succcurlyeq_X b$ if $\forall  i$  $\exists y_i \in X \cup \{b\}$ such that $a \succcurlyeq_i y_i$. The idea behind this relationship is that $a$ is closer to $X$ than $b$ is in the sense that any convex set which contains $X \cup \{b\}$ must also contain $a$. The relation $\succcurlyeq_X$ is easily seen to be reflexive and transitive. To prove antisymmetry assume $a\neq b$ are in $E\setminus X$.  Since $X\in\cgeom$ this implies $\exists i$ so $\forall x \in X,\ x \succ_i a$. For this $i$, either $b\succ_{i}a$ and thus  $a\not\succcurlyeq_{X}b$, or $a\succ_{i} b$ and therefore also $\forall x \in X,\ x \succ_i b$, the two of which imply that $b\not\succcurlyeq_{X}a$.  Hence $\succcurlyeq_{X}$ is antisymmetric and is a partial order. Now take $z$ a maximal element in $E \backslash X$ according to $\succcurlyeq_X$. If $y\not \in X\cup\{z\}$ then either $z \succ_X y$ or $z$ and $y$ are incomparable according to $\succcurlyeq_X$. In either case, there must be at least one $i$ so that $\forall x \in X \cup \{z\},\ x \succ_{i}y$ (because otherwise it would be the case that $y \succcurlyeq_X x$). Thus, $X \cup \{z\} \in \cgeom$.  This verifies~(3) of Definition~\ref{defn:cgeom} and completes the proof.
\end{proof}
The following converse was proved in~\cite{EJ}.

\begin{theorem}[\protect{\cite{EJ} Theorem 5.2}]\label{thm:EJ}
If $\cgeom$ is a convex geometry on $E$ then there are orderings $\succcurlyeq_{i}$ on $E$ such that $\cgeom$ is obtained  as in Definition~\ref{defn:cgeomfromordering}.
\end{theorem}

\section{New Proof and Dimensional Bound of a Generalized Convex Shelling}\label{sec:newproof}

The purpose of this section is to give a proof of Theorem~\ref{thm:KNO} from Theorem~\ref{thm:EJ}.  We first define a map which realizes the orderings as the orders on coordinate directions in $\mathbb{R}^{M}$.  To do so, for each $i$ arrange $E$ using the $i^{\text{th}}$ order as $x_{i1}\succ_{i}x_{i2}\succ_{i}\dotsm\succ_{i}x_{iN}$ and for $x\in E$ let $j_{i}(x)$ be the unique choice such that $x=x_{ij_{i}(x)}$. That is, $j_i(x)$ denotes $x$'s place according to the $i^{\text{th}}$ ordering. Then define $F_{i}:E\to\mathbb{R}$ by $F_{i}(x)=-(M+1)^{j_{i}(x)}$ and let $F(x)=(F_{1}(x),\dotsc,F_{M}(x)):E\to\mathbb{R}^{M}$.  We have replicated the orderings $\succcurlyeq_{i}$ from $E$ using the coordinate directions on $F(E)$, so the following is obvious.
\begin{lemma}
On $\mathbb{R}^{M}$ define $(x_{1},\dotsc,x_{M})>_{i}(y_{1},\dotsc,y_{M})$ to mean that $x_{i}>y_{i}$.  Then $\{\geq_{i}\}$ are orderings on $F(E)$ and the convex geometry $\cgeom_{1}$ they generate is isomorphic to $\cgeom$ on $E$.
\end{lemma}

There is a hull operation for $\cgeom_{1}$ which we call $\Pos$ (an abbreviation of positive sector).  If $P\subset F(E)$  then $\Pos(P)=\{x: \forall i\, \exists p(i)\in P \text{ with } x\geq_{i}p(i) \}$.  Clearly $\Pos(P)\cap F(E)\in\cgeom_{1}$ and $P\in\cgeom_{1} \iff P=\Pos(P)\cap F(E)$.  We will also need another hull operation, which we define by $\ExtHull(P)=\{x:\exists p\in\Hull(P)\text{ with } x\geq_{i}p,\ \forall i\}$, from which we can define a collection $\cgeom_{2}$  of subsets of $F(E)$ by $P\in\cgeom_{2}\iff P=\ExtHull(P)\cap F(E)$.  It is not immediately clear whether $\cgeom_{2}$ is a convex geometry, but in fact we have the following.

\begin{theorem}\label{thm:exthull=pos}
$\Pos(P)\cap F(E)=\ExtHull(P)\cap F(E)$ for any $P\subset F(E)$.  Equivalently, $\cgeom_{1}=\cgeom_{2}$.
\end{theorem}
\begin{proof}
Let $p(i)\in P$ be such that $p(i)_{i}=\min_{p\in P} p_{i}$, where the subscript denotes the $i^{\text{th}}$ coordinate. That is, $p(i)$ is the minimally-ranked member of $P$ in coordinate $i$. There is exactly one such $p(i)$ for each $i$, because $>_{i}$ is an ordering on $F(E)$,  however it is possible that $p(i)=p(j)$ for some $i\neq j$.  It is clear that $\Pos(P)=\{x:\forall i, x_{i}\geq p(i)_{i}\}$.

If $x\in\ExtHull(P)$ there are elements $p^{k}\in P$ so $x_{i}\geq \sum_{k} \alpha_{k}(p^{k})_{i}$, where $\alpha_{k}\geq0$, $\sum \alpha_{k}=1$.  But then $(p^{k})_{i}\geq p(i)_{i}$, so  $x_{i}\geq p(i)_{i}$ for each $i$, and thus $x\in\Pos(P)$ and the inclusion $\supset$ is proved.

For the converse, let $p=M^{-1}\sum_{i}p(i)\in\Hull(P)$.  Recall from the definition of $F(E)$ that $p(i)_{j}< -M$ for all $i,j$, so that $Mp_{j} = \sum_{i \neq j} p(i)_j+  p(j)_j< -M(M-1)+p(j)_{j}$.   Now let \mbox{$x\in\Pos(P) \cap F(E)$}.  If $x\in P$ then $x \in \ExtHull(P)$, so we are done.  If not, then for all $j$, the orderings $\succcurlyeq_{j}$ are antisymmetric, and so we have that $x\succ_{j}p(j)$ and therefore $Mx_{j}\ > p(j)_{j}$. \linebreak Combining this with our first estimate yields $Mp_{j} < -M(M-1)+Mx_{j}$, so that $p_{j}\leq x_{j}$.  As this is true for all $j$, it is the case that $x\geq_{j}p$ for all $j$ and thus $x\in\ExtHull(P)$.
\end{proof}

Finally we relate $\ExtHull(P)$ to the geometry given by a generalized convex shelling. Specifically, let $e_{i}$ be the $i^{\text{th}}$ coordinate vector in $\mathbb{R}^{M}$, let $Q$ be the set of points $\{0\}\cup\{\lambda e_{i}: 1\leq i\leq M\}$ where $\lambda$ is a positive real number. Furthermore, let $\cgeom_{3}$ be the convex geometry on $F(E)$ given by the convex shelling of $F(E)$ with respect to $Q$ as in Definition~\ref{defn:cvxshelling}. This is legitimate because $\Hull(Q)$ is in the positive sector, so cannot intersect $F(E)$.

\begin{theorem}\label{thm:hullshadow=pos}
There exists $\lambda$ large enough such that $\Hull(P\cup Q)\cap F(E)=\ExtHull(P)\cap F(E)$ for any $P\subset F(E)$. Equivalently $\cgeom_{2}=\cgeom_{3}$.
\end{theorem}
\begin{proof}
Since $P\subset F(E)$ is in the negative sector, then all points of $Q$ are in $\ExtHull(P)$, and thus $\Hull(P\cup Q)\subset\ExtHull(P)$. It remains to prove the converse for the intersection with $F(E)$.  In  the case $M=1$ this is clear with $Q=\{0\}$, i.e. $\lambda=0$, so we may assume $M\geq2$.

If $x\in F(E)\cap\ExtHull(P)\setminus P$ then there is $y\in\Hull(P)$ so that $y_j \leq M^{-1}\left(-(M-1)(M+1)+Mx_j)\right)$ $= -\frac{M^2-1}{M} + x_j < -1 + x_j$ for all $j$. Notice that we can take $y = M^{-1} \sum_i p(i)$ as in the previous theorem. Moreover, as $y_j \geq -(M+1)^M$, it follows that $(M+1)^{-M} y_j \geq -1$ and
\begin{equation*}
	(1-(M+1)^{-M})y_j = y_j - (M+1)^{-M} y_j \leq y_j + 1 < x_j.
\end{equation*}

Define $u_{j}=x_{j}-(1-(M+1)^{-M})y_{j}> 0$ and make these the components of $u$, so that \linebreak $x=(1-(M+1)^{-M})y+u$ with $y\in\Hull(P)$.  Thus $x$ is a convex combination of $y$ and $(M+1)^{M}u$.  We know $(M+1)^{M}u$ is in the positive sector so it is clear that for $\lambda$ large enough, $(M+1)^{M}u\in\Hull(Q)$, thus completing the proof. Since there are a finite number of convex geometries on $X$, there exists a uniformly large enough $\lambda$. But, a construction is of interest and in particular $u_{j}\leq -M -(1-(M+1)^{-M})(-(M+1)^{M})< (M+1)^{M}$, so $\lambda=(M+1)^{M+1}$ suffices.
\end{proof}

Evidently Lemma~\ref{lem:LisomL1}, Theorem~\ref{thm:exthull=pos} and Theorem~\ref{thm:hullshadow=pos} together show that for a suitable choice of $\lambda$ we have our original geometry $\cgeom$ is isomorphic to $\cgeom_{1}=\cgeom_{2}=\cgeom_{3}$, with this last being a generalized convex shelling.  Therefore, from Theorem~\ref{thm:EJ} we then have Theorem~\ref{thm:KNO}. Furthermore, one may notice from the proof that the coordinates of the embedding may be taken to be integral, so that a generalized convex shelling into $\mathbb{Z}^N$ is obtained.

\begin{definition} For a convex geometry $\cgeom$, its geometric dimension $\dim(\cgeom)$ is defined as the lowest dimension for which the convex geometry is isomorphic to a generalized convex shelling.
\end{definition}

\begin{definition}
(\citet*{ES}, Section 2) For a convex geometry $\cgeom$, its convex dimension  $\cdim(\cgeom)$ is defined as the least number of orderings with which it can be represented as in Definition \ref{defn:cgeomfromordering}.
\end{definition}

The proofs of Lemma~\ref{lem:LisomL1} and Theorems~\ref{thm:exthull=pos} and~\ref{thm:hullshadow=pos}  together show that any convex geometry $\cgeom$ represented by $k$ orderings may be embedded into $\mathbb{R}^{k}$ as a generalized convex shelling. On the other hand, the proof of Theorem~2.8 in~\cite{KNO} presents an embedding into $\mathbb{R}^{|E|}$. Therefore, we have the following corollary.

\begin{corollary}\label{cor:dimbound} For any convex geometry $\cgeom$ over $E$, $\dim(\cgeom) \leq \min(|E|,\cdim(\cgeom))$.
\end{corollary}

The upper bounds in the corollary are trivially optimal when $\cdim(\cgeom)=1$. The following examples show that for $\cdim(\cgeom)>1$ it is possible for the bound $\dim(\cgeom) \leq \cdim(\cgeom)$ to be optimal  or very far from optimal.  In the first example the bound $\dim(\cgeom) \leq \cdim(\cgeom)$ from our construction is much better than that from the argument in~\cite{KNO}.

\begin{example}[Optimal dimension bound from number of orderings]
Let $E=\{a,b,c\}$ and $\cgeom=\{\emptyset,\{a\},\{a,b\},\{a,c\},E\}$. Observe that $\cgeom$ can be represented by the two orderings $c \prec_{1} b \prec_{1} a$ and $b \prec_{2} c \prec_{2} a$, thus $\dim(\cgeom)\leq2$.  If it were possible to embed $E$  in $\mathbb{R}$ so as to obtain $\cgeom$ from the generalized convex shelling with respect to a set $Q$ we would have both $\Hull(Q)\cap E=\emptyset$ and $a\in\Hull(\{b\}\cup Q)\cap\Hull(\{c\}\cup Q)$.  It follows easily that if $I_{1}$ and $I_{2}$ are the components of $\mathbb{R}\setminus\{a\}$ then $\emptyset\neq Q\subset I_{1}$ and $\{b,c\}\subset I_{2}$, so without loss of generality we may suppose that $q<a<b<c$ for all $q\in Q$.  Then $\Hull(\{c\}\cup Q)\cap E=E$ in contradiction to the fact that $\{a,c\}$ is convex.  We conclude that $\cgeom$ cannot be given by a generalized convex shelling in $\mathbb{R}$ and therefore $\dim(\cgeom)=2$.  Comparing this to Corollary~\ref{cor:dimbound} we see that in this case $2=\dim(\cgeom)=\cdim(\cgeom)<|E|=3$.

The above construction can be generalized to obtain a family of examples with $2=\dim(\cgeom)=k$ and arbitrarily large finite $|E|\geq3$: Take  $E \supseteq \{a,b, c\}$ and let $\prec_1$, $\prec_2$ to be any extension of the orderings given above over $\{a,b,c\}$. Then the convex geometry induced by $\prec_1$, $\prec_2$ as in Definition~\ref{defn:cgeomfromordering} has dimension at most $2$ by Corollary~\ref{cor:dimbound} and at least $2$ by the previous argument, yet $|E|$ may be arbitrarily large.
\end{example}

\begin{example}[Non-optimal dimension bound from number of orderings]
Let $E$ be a finite subset of the unit circle in $\mathbb{R}^2$ with the standard convexity, so $\cgeom$ is the power set $2^{E}$. If $|E| \geq 3$, then $\dim(\cgeom)=2$ (and otherwise $\dim(\cgeom)=1$).  Suppose $\cgeom$ is represented by orderings $\prec_1,\ldots,\prec_k$.  If $x\in E$ then $E\setminus\{x\}$ is convex, so from Definition~\ref{defn:cgeomfromordering} there is $j_{x}$ so that $x$ is minimal with respect to $\prec_{j_{x}}$. Additionally, if a set of orderings minimally ranks each element, then this set represents $\cgeom$. Thus, $\cdim(\cgeom)=|E|$, which may be arbitrarily large in comparison to $\dim(\cgeom)$. 
\end{example}

\begin{example}[Lower-bounded finite lattices]
As mentioned in the introduction, two papers \cite{A} and \cite{WS} found that lower-bounded lattices can be generalized convex shelled into $\mathbb{Q}^n$. To understand whether an atomic finite convex geometry $\cgeom$ on $E$ is lower-bounded, the relation $D$ is defined on the elements of $E$ by $a D b$ if there is a set $a \in A \subseteq E$ so that the convex hull of $A$ contains $b$ and the convex hull of $A\backslash a$ does not. The lower-bound condition states that $D$ is acyclic. Notice that this restriction prevents the convexity configuration of four points on a line $a-b-c-d$, because there $b D c$ (because $c$ is in the convex hull of $b,d$ and not $d$) and likewise $c D b$. One convex geometry which is lower-bounded can be specified by a three level poset, $E=A \cup B \cup C$ where $A,B,C$ are distinct non-empty sets and $a \prec b \prec c \ \ \forall a \in A, b \in B, c\in C$. For this poset, $G\in \cgeom$ if $G=\{x: \exists a,b \in G \text{ s.t. } a \preceq x \preceq b\}$.\footnote{A poset of height four or more with this convexity would not generate a lower bounded lattice.} According to the construction of \cite{WS}, the dimension of the embedding of $G$ is the number of $D$-chains, which is $|B|(|A|+|C|) + |E|$. So, for example, if $|A|=|B|=|C|=n/3$, then the dimension of the embedding is $2n^2/9+n$. This is larger than the $n$ embedding of \cite{KNO}. In fact, one may also notice that $\cdim(\cgeom)=\max(|B|,|C|)+\max(|A|,|B|)=2n/3$ (every ordering is of the type $A > B > C$ where each member of $B,C$ must be relatively minimal or of the type $C > B > A$, again with minimality restrictions). In general, for poset convex geometries, if the poset's width is less than half of all elements, then $\cdim(\cgeom) < |E|$.
\end{example}

In \cite{ES}, the authors study the lattice of convex sets in a convex geometry using various lattice-theoretic notions of dimension, including the convex dimension $\cdim$ described above. In this paper we have related the lattice-theoretic notion of $\cdim$ with the geometric dimension notion $\dim$ due to \cite{KNO}. From the examples it is apparent that these dimensional notions are generally distinct although there are cases where the two values coincide. One question of interest would be a characterization of such geometries.  It might also be interesting to know whether there are cases where $\dim(\cgeom)=|E|\neq1$.

\section{Embeddings of Convex Geometries as Convex Polygons}

The idea of giving a concrete realization of a convex geometry as an easily visualized set of shapes in $\mathbb{R}^2$ is due to \cite{C}, who views this as a important special case of an embedding of a convex geometry.  A particular case is an embedding as convex polygons in $\mathbb{R}^{n}$.

\begin{definition}
Let $\Poly_{n}$ denote the set of convex polygons with vertices in $\mathbb{R}^{n}$.  An embedding of a convex geometry $\cgeom$ as convex polygons in $\mathbb{R}^{n}$ is a map $\F:\cgeom\to\Poly_{n}$ such that 
\begin{equation}\label{defn:polycgeom}
\cgeom = \emptyset \cup \{X \subseteq E: \forall y \in E, \F(y) \subseteq \Hull(\F(X)) \Rightarrow y \in X\}.
\end{equation}
For notational simplicity, in what follows we will write $\Poly$ for $\Poly_{2}$.
\end{definition}
\begin{remark}
 Note that similar definitions can be given for other families of shapes.  For example, $\F$ is an embedding of $\cgeom$ as circles in $\mathbb{R}^{2}$ if the above definition holds with $\Poly_{n}$ replaced by the set $\mathcal{C}$ of circles in $\mathbb{R}^{2}$.
\end{remark}

The results of the previous section prove that a convex geometry with $\cdim=n$ can be embedded as convex polygons in $\mathbb{R}^{n}$ by $x \rightarrow \Hull(F(x) \cup Q)$.  In this section we improve on that result by showing that any finite convex geometry can be embedded as convex polygons in $\mathbb{R}^2$.  This result is in the spirit of~\cite{C}, who proves that convex geometries of dimension 2 may be embedded as a finite, separated, concave set of collinear circles in $\mathbb{R}^2$ and that such an embedding characterizes these geometries.\footnote{Higher dimensional convex geometries may also be embeddable as circles, but it is not known if this is generally true. Whether three-dimensional convex geometries are embeddable as circles is a fascinating open question, specifically (4.6) of \cite{C}.}  From a special case of our argument we also obtain Corollary~4.6 of \cite{C}, which is that any convex geometry of dimension $\leq 2$ can be embedded as intervals in $\mathbb{R}$. Intervals in $\mathbb{R}$ are a common object of study between the present paper and \cite{C} as they are both convex polygons and circles.

\begin{lemma}\label{lem:polycgeom}
If $\F:E\to\Poly$ define $\cgeom_{\F}$ to be the right side of~\eqref{defn:polycgeom}.  If $\F$ is strongly injective, meaning that if $x\neq y$ then $\F(x)$ and $\F(y)$ have no common extreme points (i.e. vertices of their convex polygons) then $\cgeom_{\F}$ is a convex geometry.
\end{lemma}
\begin{proof}
Evidently both $\emptyset$ and $E$ are in $\cgeom_{\F}$, so we check conditions 2 and 3 of Defintion~\ref{defn:cgeom}. For condition 2, let $X,Y\in\cgeom_{\F}$ and $z\in E$.  If $\F(z) \subseteq\Hull(\F(X \cap Y)) \subseteq \Hull(\F(X)) \cap \Hull(\F(Y))$, then $z \in X \cap Y$. Thus $X \cap Y\in\cgeom_{\F}$ as required.

To check condition 3 it is easier to use the equivalent anti-exchange property. Suppose $X\in\cgeom_{\F}$ and there are distinct $y,z\not\in X$. Then $\Hull(\F(X)\cup\F(y)\cup\F(z))\not\subset\Hull(\F(X))$ so has a vertex $p$ from either $\F(y)$ or $\F(z)$.  If in addition $\F(z)\subset\Hull(\F(X)\cup\F(y))$ then the fact that $\F(y)$ and $\F(z)$ have no common vertices implies $p\in\F(y)$ and thus $\F(y)\not\subset\Hull(\F(X)\cup\F(z))$, verifying anti-exchange.
\end{proof}

\begin{theorem}\label{thm:PolygonEmbedding}
Any convex geometry may be embedded as convex polygons in $\mathbb{R}^2$.
\end{theorem}

\begin{proof}
A convex geometry of convex dimension $1$ is specified by a single order $x_1 \succ \ldots \succ x_n$. This convex geometry can be embedded in $\mathbb{R}$ as a nested set of intervals where $\F(x_1) \subset \ldots \subset \F(x_n)$.

Take a convex geometry of convex dimension $n\geq 2$ and for $1\leq i\leq n$ let the function \mbox{$j_{i}:E\to\mathbb{N}$} be the ranking according to the $i^{\text{th}}$ order as in Section~\ref{sec:newproof}. Formally, recall that for each ordering $\succ_i$, the set $E$  can be arranged as $x_{i1}\succ_{i}x_{i2}\succ_{i}\dotsm\succ_{i}x_{iN}$ and $j_{i}(x)$ is the unique choice such that $x=x_{ij_{i}(x)}$. That is, $j_i(x)$ denotes $x$'s place according to the $i^{\text{th}}$ ordering.  For each $1\leq i \leq n$, let \mbox{$v_{i}= (\cos(2\pi i/n),\sin(2\pi i /n)) \in\mathbb{R}^{2}$} and define $F_i(x) =(C_{n}(E)+ j_i(x) )v_{i}$, where \linebreak \mbox{$C_{n}=2|E|/(|\sec(2\pi/n)|-1)$}.  Each $F_i(x)$ is a point on a ray at angle $2 \pi i /n$  and will define the $i^{\text{th}}$ vertex for the convex polygon $\F(x)$. That is, we set $\F(x) = \Hull(F_1(x),\ldots,F_n(x))$.  Clearly the origin $(0,0)\in\Hull(\F(X))$.


We now show  $\cgeom_{\F}=\cgeom$.  Suppose $X\subset E$ is not in $\cgeom$.  Then there is $z \notin X$ so that for each $i$, there exists $x_i \in X$ with $z \succ_i x_i$.  Thus $j_{i}(x_{i})>j_{i}(z)$ and $|F_i(x_i)| > |F_i(z)|$ for all $i$.  Using that $(0,0)\in\F(x)$ it easily follows that $\F(z)= \Hull\bigl(\{F_1(x_1),\dotsc,F_n(x_n)\}\bigr) \subseteq\Hull(\F(x_1)\cup\dotsm\cup\F(x_n)) \subseteq \Hull(\F(X))$.  But $z\notin X$ and $F(z)\subseteq\Hull\F(X)$ implies $X\notin \cgeom_{\F}$, so $\cgeom_{\F}\subset\cgeom$.

Conversely, take a set $X \in\cgeom$ and $z\notin X$. Then there is some $i$ so that for all $x\in X$ we have $x \succ_i z$ and therefore $j_{i}(z)\geq j_{i}(x)+1$.  Rotate so $i=n$ and let $r=\max\{|F_{n}(x)|:x\in X\}$.  It is clear that
\begin{equation*}
	\Hull(\F(X))\subset \Hull \bigl( \{ (C_{n}(E)+|E|)v_{1},\dotsc,  (C_{n}(E)+|E|)v_{n-1},rv_{n}\} \bigr),
	\end{equation*}
but the first component of any $ (C_{n}(E)+|E|)v_{k}$ does not exceed $(C_{n}(E)+|E|)\cos(2\pi/n)\leq \linebreak (C_{n}(E)+|E|)|\cos(2\pi/n)|\leq C_{n}<C_{n}+j_{n}(z)$. The first component of $rv_{n}$ is $r<C_{n}+j_{n}(z)$, so $F_{n}(z)\not\in\Hull(\F(X))$ and $\F(z)\not\subset\Hull(\F(X))$.  Thus $X\in\cgeom_{\F}$ and $\cgeom=\cgeom_{\F}$.
%
%
%
\end{proof}

The proof of the above theorem demonstrates that any convex geometry of dimension $2$ may be embedded into the real line as intervals. 

\newpage

\bibliography{bibfile}
\bibliographystyle{plainnat}

\end{document}